\documentclass[12pt]{article}
\usepackage{amsmath,amssymb,amsfonts,amsthm}
\usepackage{multirow}
\usepackage{rotating}
\usepackage{paralist}
\topmargin- 1.4cm
 \small
{\par\vskip8pt minus3pt\rm}
\newcounter{item}[section]
\newcounter{kirshr}
\newcounter{kirsha}
\newcounter{kirshb}
\newtheorem{theorem}{Theorem}[section]

\newtheorem{corollary}[theorem]{Corollary}
\newtheorem{proposition}[theorem]{Proposition}

\newtheorem{remark}[theorem]{Remark}

\newtheorem{definition}[theorem]{Definition}

\newcommand{\Ahmed}[2]{\stackrel{\,#2}{#1}\ \!\!\!}

\newcommand\overcirc[1]{\stackrel{\tiny\circ}{#1}}
\newcommand\undersym[2]{\raisebox{-6pt}{\tiny$#2$}{\kern-5pt}\mbox{$#1$}}

\def\Section#1{\vspace{30truept}\addtocounter{section}{1}\setcounter{theorem}{0}\setcounter{equation}{0}
{\noindent\Large\bf\arabic{section}.~~#1}\par \vspace{12pt}}

\begin{document}
\title{{\bf On Finslerized Absolute Parallelism\\ Spaces}\footnote{arXiv number: 1206.4505 [math.DG] } \footnote{This paper has been presented in the 8th. International Conference on \lq \lq Finsler Extensions of Relativity Theory (FERT2012)", 25 June-01 July 2012, Moscow- Fryazino, Russia.}}
\author{{\bf{Nabil L. Youssef}$^{1 ,2}$, Amr M. Sid-Ahmed$^{1}$ and {Ebtsam H. Taha}$^{1 ,2}$}}
\date{}
\maketitle                     
\vspace{-1cm}
\begin{center}
{$^{1}$Department of Mathematics, Faculty of Science,\\ Cairo
University, Giza, Egypt}\\
\vspace{0.3cm}
{$^{2}$Center for Theoretical Physics (CTP)\\ at the British University in Egypt (BUE)}
\end{center}
\vspace{-0.5cm}
\begin{center}
$
\begin{array}{ll}
\text{Emails:}&\text{nlyoussef@sci.cu.edu.eg, nlyoussef2003@yahoo.fr} \\
  &\text{amr@sci.cu.edu.eg, amrsidahmed@gmail.com} \\
  &\text{ebtsam.taha@sci.cu.edu.eg, ebtsam.h.taha@hotmail.com}
\end{array}
$
\end{center}

\vspace{1cm} \maketitle
\smallskip


\noindent{\bf Abstract.} The aim of the present paper is to
construct and investigate a Finsler structure within the framework of a Generalized Absolute Parallelism
space (GAP-space). The Finsler structure is obtained from the vector fields forming the parallelization of the GAP-space. The resulting space, which we refer to
 as a Finslerized Absolute Parallelism (Parallelizable) space, combines within its geometric structure the simplicity of GAP-geometry and the richness of Finsler geometry, hence is potentially more suitable for applications and especially for describing physical phenomena.
  A study of the geometry of the two structures and their interrelation is
carried out. Five connections are introduced and their
 torsion and curvature tensors derived. Some special Finslerized Parallelizable spaces are singled out. One of the main reasons  to introduce this new space is that both Absolute Parallelism and Finsler geometries have proved effective in the
formulation of physical theories, so it is worthy to try to build a more general geometric structure that would share the benefits of both geometries.

\bigskip


\bigskip


\vspace{3pt}
\noindent{\bf PACS 2010}: 02.04.Hw, 45.10.Na, 04.20.-q, 04.50.-h.

\newpage


\Section{Introduction and Motivation }
\hspace{12pt}
The philosophy of geometrization of physics has been advocated by Albert Einstein early in the twentieth century. He has successfully used this philosophy to capture the notion of the gravitational field in the language of four dimensional Reimannian geometry. On the one hand, the degrees of freedom in this case are ten and on the other hand it is shown that gravity is well described by ten field variables, namely, the components of the metric tensor. Extending this philosophy to deal with other interactions, besides the gravity, one needs to increase the degrees of freedom of the geometry used. Two main strategies arise: either to increase the dimension of the Riemmannian space considered or to deal with another wider geometry enjoying more degrees of freedom (without increasing the dimension). Our choice is the second strategy. Theories constructed in such geometries (cf. for example \cite{Mikhail}, \cite{Moller},  \cite{quant} and \cite{quant1}) show the advantages of using geometries with extra degrees of freedom. This is clearly manifested through physical applications of such theories (cf. \cite{*} and \cite{**}).
\bigskip

Finsler geometry is a generalization of Riemmannian geometry. It is much wider in scope and richer in content than Riemannian geometry.
In the rich arena of Finsler geometry,
geometric objects depend in general on both the positional argument $x$ and the directional argument $y$.\,This
is a crucial difference between Finsler geometry and Reimannian geometry in which the geometric
objects of the latter depend only on the positional argument $x$. Moreover, unlike Riemannian geometry which
has one canonical linear connection defined on $M$, Finsler geometry admits (at least) four Finsler
connections: Cartan, Berwald, Chern and Hashigushi  connections. However, these are not defined on the
manifold $M$ but are defined on the pull-back bundle $\pi^{-1}(TM)$ (the pullback bundle of the tangent bundle $TM$ by $\pi : TM \longrightarrow M$) \cite{SG} and \cite{Vector bundles}. Finally, in the context of
Riemannian geometry, there is only one curvature tensor corresponding to the canonical connection and
no torsion tensor. In Finsler geometry, for each connection defined corresponds three curvature tensors and five torsion tensors.

\bigskip

  The geometry of parallelizable manifolds or absolute parallelism geometry (AP-geometry) has many
advantages in comparison to Riemannian geometry. As opposed to Reimannian geometry,
which admits only one symmetric linear connection, AP-geometry admits at least four (built in)
linear connections, two of which are non-symmetric and three of which have non-vanishing
curvature tensors \cite{AMR}. Last, but not least, associated with an AP-space, there is a Riemannian
structure defined in a natural way. Consequently, AP-geometry contains within its geometric structure
all the mathematical machinery of Riemannian geometry.

\bigskip

The GAP-geometry \cite{GAP} is a generalization of the classical AP-geometry. As in the case of Finsler geometry, the geometric objects of the
GAP-space live in the pull-back bundle $\pi^{-1}(TM)$. Accordingly, they depend
in general on both the positional argument $x$ and the directional argument $y$. Many
geometric objects which have no counterpart in the classical AP-geometry emerge in this more general context.

\bigskip

Since both the AP and Finsler geometries have been successfully used in describing physical phenomena, it is natural and useful to construct a more general
geometric structure that would enjoy the advantages of both geometries. To accomplish this, we construct what we refer to as a Finsleraized Absolute Parallelism space or a Finslerized Parallelizable space (FP-space). Due to its richer and wider geometric structure compared to both
AP-geometry and Finsler geometry, an FP-space  may be a better candidate for the
attempts of unification of fundamental interactions. This is because an FP-space has more degrees of freedom than either AP-space or
Finsler space; both spaces already being used in the process of unification and solving physical problems.

\bigskip

In the present paper, we construct a Finsler structure on a GAP-space. The paper consists mainly of two (related) parts; Finslerized parallelizable  spaces and some  special Finslerized Parallelizable spaces. The paper is organized in the following manner. In section 1, following the introduction, we give a short survey of the basic concepts of
Finsler geometry. In section 2,  we give a brief account of the GAP-geometry. In section 3, we introduce a Finsler structure on the GAP-space defined in terms of the $\pi$-vector fields $\lambda$'s forming the parallelization of the GAP-space. We show that the Finsler metric coincides with the GAP-metric. This gives a strong link between the two geometric structures. A study of the geometric objects of the
resulting Finsler structure and the geometric objects of the GAP-space is accomplished. We show that all geometric objects of an FP-space can be expressed in terms of the $\lambda$'s (including the nonlinear connection). In section 4, some special Finslerized Parallelizable spaces are studied, namely, Landsberg, Berwald, Minkowskian and Riemannian spaces. Some interesting results concerning these spaces are obtained and the relationship between them is investigated. We end the paper with some concluding remarks.


\Section{General review on Finsler geometry}

 We give here a brief account of Finsler geometry. Most of the material presented here may be found in \cite{FB} and \cite{SG}. For more details we refer, for example, to \cite{Bejancu}, \cite{R}, \cite{MC} and \cite{TS}. \par
Let $M$ be a differentiable manifold of dimension $n$ and of class $C^{\infty}$. Let $\pi:TM\to M$
be its tangent bundle. If $(U,\, x^{\mu})$ is a local chart on $M$,
then ($\pi^{-1}(U), \ (x^{\mu},\, y^{\mu}$)) is the corresponding
local chart on $TM$. The coordinate transformation law on $TM$ is
given by:
$$x^{\mu'} = x^{\mu'}(x^{\nu}), \ \ \ y^{\mu'} = p^{\mu'}_{\nu} \, y^{\nu},$$
where $p^{\mu'}_{\nu} := \frac{\partial {x^{\mu'}}}{{\partial
x^{\nu}}}$ and $\text {det} (p^{\mu'}_{\nu})\neq 0.$
\begin{definition}
A generalized Lagrange space is a pair $\textbf{GL}^n :=(M,\, g(x,\,y))$, where $M$ is a smooth $n$-dimensional manifold and $g(x,\,y)$ is a symmetric non degenerate tensor field of type $(0,2)$ on TM.
\end{definition}

\begin{definition}
A Lagrange structure on a manifold $M$ is a mapping ${L}: TM \longrightarrow \mathbb{R}$ with the following properties:
\begin{description}
    \item [(a)] ${L}$ is $C^{\infty}$ on {\fontsize{18}{10}{$\tau$}}$M$ $:= TM\setminus \{0\}$ and $C^{0}$ on TM,
    \item [(b)] The Hessian of ${L}$, having components $${g_{\alpha \beta}}(x,\,y) := \frac{1}{2} {\dot{\partial}_ {\alpha}  \dot{\partial}_ {\beta}} {L},$$
is non degenerate  at each point of {\fontsize{18}{10}{$\tau$}}$M$, where $\dot{\partial}_\mu := \frac{\partial}{\partial y^\mu}$.
  \end{description}
The pair $\textbf{L}^{n}:=(M,\,L(x,\,y))$ is called a Lagrange manifold or a Lagrange space.
\end{definition}
\begin{definition}
A Finsler structure on a manifold $M$ is a mapping $F : TM \longrightarrow \mathbb{R}$ with the following properties:
  \begin{description}
    \item [(a)] $F$ is $C^{\infty}$ on {\fontsize{18}{10}{$\tau$}}$M$ and $C^{0}$ on TM,
    \item [(b)] For all $ (x,y)\in TM,\, F(x,\,y) \ge 0 \, \text{ and } \, F(x,\,y) = 0 \, \Longleftrightarrow \,y=0 $,
    \item [(c)] $F$ is positively homogenous of degree one in y; that is, $F(x,\,ay)= aF(x,\,y) \, \, \forall a >0,$
    \item [(d)] The Hessian of $ {F}^{2}$, having components $${g_{\alpha \beta}}(x,\,y) := \frac{1}{2} {\dot{\partial}_ {\alpha}  \dot{\partial}_ {\beta}} {F} ^{2},$$
is positive definite at each point of {\fontsize{18}{10}{$\tau$}}$M$.
  \end{description}
\end{definition}
The pair $\textbf{F}^{n}:=(M,\, F(x,\,y))$ is called a Finsler manifold or a Finsler space and the symmetric bilinear form  ${g_{\alpha \beta}}(x,\,y)$ is called the Finsler metric tensor of  the Finsler space $\textbf{F}^{n}$.
\begin{remark}
\em{ If ${g_{\alpha \beta}}(x,\,y)$ doesn't depend on $y^i$, then $\textbf{F}^n$ is a Riemannian space $\textbf{R}^n$. In fact, we have the following sequence of proper inclusions:
$$\{\textbf{R}^n\} \subset \{\textbf{F}^n\} \subset \{\textbf{L}^n\} \subset  \{\textbf{GL}^n\}.$$}
\end{remark}
\begin{definition} A nonlinear connection $N$ on $TM$ is a system of $n^{2}$ functions
$N^{\alpha}_{\beta}(x, \ y)$, defined on every local chart of TM, transforming according to the law
\begin{equation}N^{\alpha'}_{\beta'} = p^{\alpha'}_{\alpha} \, p^{\beta}_{\beta'} \, N^{\alpha}_{\beta} +
p^{\alpha'}_{\mu} \,
p^{\mu}_{\beta'\sigma'} \, y^{\sigma'}\end{equation}
under any change of coordinates $(x^{\mu}) \longrightarrow (x^{\mu'})$, where
$p^{\mu}_{\beta'\sigma'} := \frac{\partial
{p^{\mu}_{\beta'}}} {\partial x^{\sigma'}} = \frac{\partial^{2}
x^{\mu}}{\partial x^{\beta'}\partial x^{\sigma'}}.$
\end{definition}
The nonlinear connection $N$ leads to the direct sum decomposition
\begin{center}
$T_{u}(TM) = H_{u}(TM)\oplus V_{u}(TM) \ \ \forall u\in \, TM,$
\end{center}
where $H_u(TM)$ is the {\it horizontal} space at $u$ associated with
$N$ supplementary to the {\it vertical} space $V_{u}(TM)$. If
$$\delta_{\mu}: = \partial_{\mu} - N^{\alpha}_{\mu}\dot
{\partial}_{\alpha},$$ where $\partial_{\mu}: =
\frac{\partial}{\partial x^{\mu}}$, then $(\dot{\partial_{\mu}})$ is
the natural basis of $V_u(TM)$ and  $(\delta_{\mu})$ is the natural
basis of $H_{u}(TM)$ adapted to $N$.
\begin{definition}
  A Finsler connection on $M$ is a triplet
$F\Gamma = (\textbf{F}^{\alpha}_{\mu\nu}(x, y),\,\textbf{N}^{\alpha}_{\mu}(x, y),\,
\textbf{C}^{\alpha}_{\mu\nu}(x, y))$, where $\textbf{F}^{\alpha}_{\mu\nu}(x,\,y)$ transform as the
coefficients of a linear connection, $\textbf{N}^{\alpha}_{\mu}(x, y)$ transform as the
coefficients of a
nonlinear connection and $\textbf{C}^{\alpha}_{\mu\nu}(x,\,y)$
transform as the components of a tensor field of type $(1,2)$.
\par
The horizontal ($h$-) and vertical ($v$-) covariant derivatives of a $(1,1)$-tensor field
$X^{\alpha}_{\mu}$
with respect to the Finsler connection $F\Gamma = (\textbf{F}^{\alpha}_{\mu \nu},\, \textbf{N}^{\alpha}_{\mu},\, \textbf{C}^{\alpha}_{\mu \nu})$ are defined respectively by:
\begin{equation}X^{\alpha}_{\mu|\nu}: =
\delta_{\nu}X^{\alpha}_{\mu} +
X^{\eta}_{\mu} \, \textbf{F}^{\alpha}_{\eta\nu} -
X^{\alpha}_{\eta} \, \textbf{F}^{\eta}_{\mu\nu},
\end{equation}
\begin{equation}
X^{\alpha}_{\mu||\nu} := \dot{\partial}_{\nu}X^{\alpha}_{\mu} +
X^{\eta}_{\mu} \, \textbf{C}^{\alpha}_{\eta\nu} - X^{\alpha}_{\eta} \,
\textbf{C}^{\eta}_{\mu\nu}.
\end{equation}
The above rules can be evidently generalized to any tensor of arbitrary type. \end{definition}
\begin{definition}The torsion tensors of a Finsler connection $F\Gamma= (\textbf{F}^{\alpha}_{\mu \nu},\, \textbf{N}^{\alpha}_{\mu},\, \textbf{C}^{\alpha}_{\mu \nu})$ on $M$ are given by:
\begin{description}
  \item[(a)] $(h)h$-torsion tensor $\textbf{T}^{\alpha}_{\mu \nu} := \textbf{F}^{\alpha}_{\mu \nu} - \textbf{F}^{\alpha}_{\nu \mu},$
  \item[(b)] $(h)hv$-torsion tensor $\textbf{C}^{\alpha}_{\mu \nu} := \text{ the third connection coefficients } \textbf{C}^{\alpha}_{\mu \nu},$
  \item[(c)] $(v)h$-torsion tensor $\textbf{R}^{\alpha}_{\mu \nu} := \delta_{\nu} \textbf{N}^{\alpha}_{\mu} - \delta_{\mu} \textbf{N}^{\alpha}_{\nu},$
  \item[(d)] $(v)hv$-torsion tensor $\textbf{P}^{\alpha}_{\mu \nu} := {\dot{\partial}}_{\nu} \textbf{N}^{\alpha}_{\mu} -  \textbf{F}^{
  \alpha}_{\mu \nu},$
  \item[(e)] $(v)v$-torsion tensor $\textbf{S}^{\alpha}_{\mu \nu} := \textbf{C}^{\alpha}_{\mu \nu} - \textbf{C}^{\alpha}_{\nu \mu}.$
\end{description}
\end{definition}
\begin{definition}The curvature tensors of a Finsler connection $\textbf{F}\Gamma = (\textbf{F}^{\alpha}_{\mu \nu},\, \textbf{N}^{\alpha}_{\mu},\, \textbf{C}^{\alpha}_{\mu \nu})$ on $M$ are given by:
\begin{description}
\item[(a)] $h$-curvature tensor ${\textbf{R}^{\alpha}_{\mu \nu \sigma}} := \mathfrak{U}_{(\nu,\sigma)} \{ \delta_{\sigma} \textbf{F}^{\alpha}_{\mu \nu} + \textbf{F}^{\eta}_{\mu \nu} \, \textbf{F}^{\alpha}_{\eta \sigma}\} + \textbf{C}^{\alpha}_{\mu \eta} \, \textbf{R}^{\eta}_{\nu \sigma},$
   \item[(b)] $hv$-curvature tensor ${\textbf{P}^{\alpha}_{\mu \nu \sigma}} := \dot{\partial_{\sigma}} \textbf{F}^{\alpha}_{\mu \nu} - \textbf{C}^{\alpha}_{\mu \sigma | \nu} + \textbf{C}^{\alpha}_{\mu \eta} \, \textbf{P}^{\eta}_{\nu \sigma}, $
   \item[(c)] $v$-curvature tensor ${\textbf{S}^{\alpha}_{\mu \nu \sigma}} := \mathfrak{U}_{(\nu,\sigma)} \{ \dot{\partial}_{\sigma} \textbf{C}^{\alpha}_{\mu \nu} + \textbf{C}^{\eta}_{\mu \sigma } \, \textbf{C}^{\alpha}_{\eta \nu} \}, $
\end{description}
 where $ \mathfrak{U}_{(\nu, \sigma)} \{ A_{\nu \sigma} \} := A_{\nu \sigma} - A_{\sigma \nu}  $ is the alternate sum with respect to the indices $\nu$ and $\sigma $.
\end{definition}
\begin{definition}
For a Finsler space, $F^{n} = (M, {F})$, we have the following geometric objects:
\end{definition}
\vspace{-0.8cm}
\begin{equation}\label{Cartan tensor}\!
\textbf{(a)}\,\,\displaystyle{C_{\beta \mu \nu } := \frac{1}{2} \,
\dot{\partial}_{\beta} g_{\mu \nu} = \frac{1}{4} \,
\dot{\partial}_{\beta} \dot{\partial}_{\mu} \dot{\partial}_{\nu}
{F}^{2}}; \text{ \emph{the Cartan tensor},} \qquad \qquad \qquad \qquad
\qquad  \qquad \, \, \, \, \, \, \, \,
\end{equation}
\vspace{-0.4cm}
\begin{equation}\! \label{dot Christoffel}
\textbf{(b)}\,\, \displaystyle{C^{\alpha}_{\mu \nu}:= g^{
\alpha\eta} \, C_{\eta \mu \nu} = \frac{1}{2}
g^{\alpha\eta}(\dot{\partial}_{\mu} g_{\nu\eta} +
\dot{\partial}_{\nu} g_{\mu\eta} - \dot{\partial}_{\eta}
g_{\mu\nu})}; \text{ the } \dot{\partial}\text{-\emph{Christoffel
symbols},} \qquad \, \,\, \, \,\, \, \, \,
\end{equation}
\begin{equation}
\!\label{formal Christoffel}
\textbf{(c)}\,\,\displaystyle{\gamma^{\alpha}_{\mu \nu} :=
\frac{1}{2} \, g^{\alpha \eta} ( \partial_{\mu}g_{ \nu\eta} +
\partial_{\nu} g_{\mu \eta} -\partial_{\eta} g_{\mu \nu})}; \text{ \emph{the formal
Christoffel symbols},} \qquad \,\,\,\,\,\,\,\,\,\,\,\,\,\,\,
\end{equation}
\begin{equation}\label{delta Christoffel}\!
\textbf{(d)}\,\,\displaystyle{\Gamma^{\alpha}_{\mu\nu} := \frac{1}{2}\,
g^{\alpha\eta}(\delta_{\mu} g_{\nu\eta} + \delta_{\nu} g_{\mu\eta} -
\delta_{\eta} g_{\mu\nu})};\text{ the } \delta\text{-\emph{Christoffel
symbols}.}\qquad \qquad \qquad \qquad \qquad
\end{equation}
 \begin{theorem}
Let $(M,\,F)$ be a Finsler space.
\begin{description}
  \item[(a)]  The $n$ functions on $TM$, defined by
\begin{equation}\label{spray}
    G^{\alpha}(x,\,y):=\frac{1}{2}\gamma^{\alpha}_{\mu\nu}\,y^{\mu} \,y^{\nu},
\end{equation}are the components of a spray called the canonical spray associated with $(M,\, F)$.

  \item[(b)] The $n^2$ functions $N^{\alpha}_{\beta}$  on $TM$, defined by
\begin{equation}\label{Nonlinear}
    N^{\alpha}_{\beta}:=\dot{\partial}_{\beta}G^{\alpha},
\end{equation}
are the coefficients of a nonlinear connection called the Barthel
connection, or the Cartan nonlinear connection, associated with $(M,\,F)$.
  \item[(c)] The $n^3$ functions $G^{\alpha}_{\sigma \beta}$ on $TM$, defined by
\begin{equation}\label{berwald linear}
G^{\alpha}_{\sigma
\beta}:={{\dot{\partial}_{\beta}}N^{\alpha}_{\sigma}},
\end{equation}
are the coefficients of a linear connection, called the Berwald (linear) connection.
\end{description}
\end{theorem}
\begin{remark}
\em{ It should be noted that $\gamma^{\alpha}_{\mu \nu}$ are neither the coefficients of a linear connection nor the components of a tensor. However, it gives rise to a nonlinear connection (Barthel connection) and a linear connection (Berwald connection) as shown above.}\end{remark}
\par
In Finsler geometry there are at least four Finsler connections which are the most famous: the Cartan, Berwald, Chern and Hashiguashi connections. The nonlinear connection associated with any one of these Finsler connections coincides with the Barthel connection. The widely using connection in Finsler geometry is the Cartan connection in the sense that the other Finsler connections can be obtained from it by applying certain processes \cite{Mats}. In this paper, we deal only with two Finsler connections: the Cartan connection and the Berwald connection which are the most important ones. In the next two theorems, we state the existence and uniqueness theorems for these two Finsler connections.

\begin{theorem}
On a Finsler manifold $(M,\, F)$, there exists a unique Finsler connection $F\Gamma= (\textbf{F}^{\alpha}_{\mu \nu},\, \textbf{N}^{\alpha}_{\mu},\, \textbf{C}^{\alpha}_{\mu \nu})$ having the following properties:
 \begin{description}
   \item[$C_1:$] ${y^{\alpha}}_{| \,\beta} = 0,$ i.e. $\textbf{N}^{\alpha}_{\beta} =y^{\mu}\,\textbf{F}^{\alpha}_{\beta \mu}$,
   \item[$C_2:$] $F \Gamma$ is $h$-metric: $g_{\mu \nu |\alpha}=0$, i.e. $\displaystyle{\delta_\alpha g_{\mu \nu}=g_{\nu \beta} \textbf{F}^\beta_{\mu\alpha}+g_{\mu \beta}
     \textbf{F}^\beta_{\nu\alpha}}$,
   \item[$C_3:$] $F \Gamma$ is $v$-metric: $g_{\mu \nu ||\alpha}=0$, i.e. $\displaystyle{\dot{\partial}_\alpha g_{\mu \nu}=g_{ \nu \beta}  \textbf{C}^\beta_{\mu\alpha}+g_{\mu \beta}
     \textbf{C}^\beta_{\nu\alpha}}$,
   \item[$C_4:$] $F \Gamma$ is $h$-symmetric: $\textbf{T}^{\alpha}_{\mu \nu}=0$, i.e. $\textbf{F}^{\alpha}_{\mu \nu} = \textbf{F}^{\alpha}_{\nu \mu}$,
   \item[$C_5:$] $F \Gamma$ is $v$-symmetric: $\textbf{S}^{\alpha}_{\mu \nu}=0$, i.e. $\textbf{C}^{\alpha}_{\mu \nu} = \textbf{C}^{\alpha}_{\nu \mu}$.
 \end{description}
 This connection is called the \textbf{Cartan} connection.
 \end{theorem}
 The Cartan connection $C\Gamma$ associated with the Finsler space $(M,\, F)$  is given by
 $$C\Gamma := (\Gamma^{\alpha}_{\mu\nu} , N^{\alpha}_{\mu} , C^{\alpha}_{\mu\nu}),$$
 where $\Gamma^{\alpha}_{\mu\nu}, \, N^{\alpha}_{\mu}$ and $C^{\alpha}_{\mu\nu}$ are defined above by $(\ref{delta Christoffel}), \, (\ref{Nonlinear})$ and $(\ref{dot Christoffel})$, respectively.
\begin{theorem}
On a Finsler manifold (M,\,F), there exists a unique Finsler connection $F\Gamma= (\textbf{F}^{\alpha}_{\mu \nu},\, \textbf{N}^{\alpha}_{\mu},\, \textbf{C}^{\alpha}_{\mu \nu})$ with the following properties:
\begin{description}
  \item[$B_1:$] $y^{\alpha}_{| \, \beta} = 0,$ i.e. $\textbf{N}^{\alpha}_{\beta} =y^{\mu}\,\textbf{F}^{\alpha}_{\beta \mu}$,
  \item[$B_2:$] $F_{ | \alpha} = 0,$ i.e. $\delta_{\alpha}F =0$,
  \item[$B_3:$]  $F \Gamma$ is $h$-symmetric: $\textbf{T}^{\alpha}_{\mu \nu}=0$, i.e. $\textbf{F}^{\alpha}_{\mu \nu} = \textbf{F}^{\alpha}_{\nu \mu}$,
  \item[$B_4:$]  $\textbf{C}^{\alpha}_{\mu \nu} = 0,$
  \item[$B_5:$]  $\textbf{P}^{\alpha}_{\mu \nu} = 0$ i.e. $\textbf{F}^{\alpha}_{\mu \nu} = \dot{\partial}_{\mu} \textbf{N}^\alpha_\nu$.
\end{description}
 This connection is called the \textbf{Berwald} connection.
 \end{theorem}
 The Berwald connection $B \Gamma$ associated with the Finsler space $(M,\, F)$ is given by
$$ B \Gamma = (G^{\alpha}_{\mu \nu} , N^{\alpha}_{\mu } , 0),$$
where $G^{\alpha}_{\mu \nu}$ and $N^{\alpha}_{\mu}$ are defined by $(\ref{berwald linear})$ and $(\ref{Nonlinear})$, respectively.


\Section{Survey on generalized absolute parallelism geometry}
\hspace{12pt}
We here give a short survey on generalized absolute parallelism geometry. For more details we refer to \cite{GAP}.
\begin{definition} An $n$-dimensional manifold $M$ is called d-parallelizable,
or generalized absolute parallelism space (GAP-space), if  the
pull-back bundle $\pi^{-1}(TM)$ admits  $n$ global
independent sections ($\pi$-vector fields)\
$\undersym{\lambda}{i}(x,\,y)$, $i = 1, ..., n$.
\end{definition}
The components of the vector field $\ \undersym{\lambda}{i}(x,\,y)$ in the natural basis are denoted by $\ {\undersym{\lambda}{i}}^\alpha, \, \alpha=1,\cdots,n$. Latin indices are numbering (mesh) indices and the Greek indices are coordinate (world) indices.
The Einstein summation convention is applied on both Latin (mesh) and Greek
(world) indices, where all Latin indices are written in a lower position and in most cases, when mesh indices
appear they will be in pairs meaning summation.

\vspace{5pt}
If \ $\undersym{\lambda}{i} := (\,\undersym{\lambda}{i}^{\alpha}(x,\, y))$; \,$i,\alpha = 1, ..., n$,
then
\begin{equation}\undersym{\lambda}{i}^{\alpha} \ \undersym{\lambda}{i}_{\beta} = \delta^{\alpha}_{\beta}, \ \ \
\undersym{\lambda}{i}^{\alpha} \ \undersym{\lambda}{j}_{\alpha} =
\delta_{ij},\vspace{-5pt}\end{equation} where $( \
\undersym{\lambda}{i}\,_{\alpha})$ denotes the inverse of the matrix
$( \ \undersym{\lambda}{i}^{\alpha})$. \vspace{7pt}

\begin{theorem} A GAP-space is a generalized Lagrange space.
\end{theorem}
In fact, the covariant tensor field $g_{\mu\nu}(x, y)$ of order $2$
given by\vspace{-5pt}
\begin{equation}g_{\mu\nu}(x, y): =
{\undersym{\lambda}{i}_{\mu}}(x,\,y) \
{\undersym{\lambda}{i}_{\nu}}(x,\,y),\vspace{-5pt}\end{equation} defines a
metric in the pull-back bundle $\pi^{-1}(TM)$ with inverse given
by\vspace{-5pt}
\begin{equation}g^{\mu\nu}(x, y) ={\undersym{\lambda}{i}^{\mu}}(x,\,y) \ {\undersym{\lambda}{i}^{\nu}}(x,\,y).
\vspace{-5pt}\end{equation}
\begin{definition} Let $(M, g)$ be a generalized Lagrange space equipped
with a nonlinear connection $N^{\alpha}_{\mu}$. Then a Finsler connection $D = (F^{\alpha}_{\mu\nu},N^{\alpha}_{\mu},
C^{\alpha}_{\mu\nu})$ is said to be metrical with respect to the metric $g$ if
\begin{equation}g_{\mu\nu|\alpha} = 0, \ \ \ g_{\mu\nu||\alpha} = 0.
\end{equation}
\end{definition}
\vspace{6pt}
The following result was proved by R. Miron \cite{reiman}. It
guarantees the existence of a unique {\it torsion-free} metrical
 Finsler connection on any generalized Lagrange space equipped with a
nonlinear connection.
\begin{theorem}\label{Miron} Let $(M, g(x, y))$ be a generalized Lagrange space.
Let $N^{\alpha}_{\mu}$ be a given nonlinear connection on $TM$.
Then there exists a unique metrical Finsler connection $\overcirc{D}
= (\Ahmed{\Gamma}{\circ}\!^{\alpha}_{\mu\nu},N^{\alpha}_{\mu},
\Ahmed{C}{\circ}\!^{\alpha}_{\mu\nu})$ such that \
$\Ahmed{T}{\circ}^{\alpha}_{\mu\nu}: =
 \ \Ahmed{\Gamma}{\circ}^{\alpha}_{\mu\nu} - \ \Ahmed{\Gamma}{\circ}^{\alpha}_{\nu\mu} = 0$
and \ $\Ahmed{S}{\circ}^{\alpha}_{\mu\nu}: = \
\Ahmed{C}{\circ}^{\alpha}_{\mu\nu} -
 \ \Ahmed{C}{\circ}^{\alpha}_{\nu\mu} = 0$.\\
 \par
This Finsler connection is given by $N^{\alpha}_{\mu}$ and the generalized Christoffel symbols:
\vspace{-0.15cm}
\begin{equation}\Ahmed{\Gamma}{\circ}^{\alpha}_{\mu\nu} = \frac{1}{2}\,
g^{\alpha\epsilon}(\delta_{\mu} g_{\nu\epsilon} + \delta_{\nu}
g_{\mu\epsilon} - \delta_{\epsilon} g_{\mu\nu}),\end{equation}
\begin{equation}\Ahmed{C}{\circ}^{\alpha}_{\mu\nu} = \frac{1}{2}\,
g^{\alpha\epsilon}(\dot{\partial}_{\mu} g_{\nu\epsilon} +
\dot{\partial}_{\nu} g_{\mu\epsilon} - \dot{\partial}_{\epsilon}
g_{\mu\nu}).\end{equation}
This connection will be referred to as the \textbf{Miron} connection.
\end{theorem}
\vspace{-0.1cm}
We define another Finsler connection $D = (\Gamma^{\alpha}_{\mu\nu}, N^{\alpha}_{\mu}, C^{\alpha}_{\mu\nu})$ in terms of
$\,\overcirc{D}$ by:
\begin{equation}
\Gamma^{\alpha}_{\mu\nu}: = \ \Ahmed{\Gamma}{\circ}^{\alpha}_{\mu\nu} +\,
\undersym{\lambda}{i}^{\alpha} \ \undersym{\lambda}{i}_{\mu{o\atop{|}}\nu}, \ \ \ \ N^{\alpha}_{\mu}= N^{\alpha}_{\mu}, \ \ \ \ C^{\alpha}_{\mu\nu}: = \ \Ahmed{C}{\circ}^{\alpha}_{\mu\nu}+ \, \undersym{\lambda}{i}^{\alpha} \ \undersym{\lambda}{i}_{\mu{o\atop{||}}\nu},\vspace{-0.25cm}\end{equation}
where \lq\lq\,$\Ahmed{|}{\circ}\,$\,\rq\rq\, and
\lq\lq\,$\Ahmed{||}{\circ}\,$\,\rq\rq\, denote the $h$- and $v$-covariant
derivatives with respect to the Miron connection
$\,\Ahmed{D}{\circ}$ . The Finsler connection $D$ is called the \textbf{canonical} connection.
\par
 If \lq\lq\,$|$\,\rq\rq\, and
\lq\lq\,$||$\,\rq\rq\, denote the $h$- and $v$-covariant derivatives
with respect to the canonical connection $D$, then one can show that
\begin{equation}{\lambda}^{\alpha}\!\, _{|\mu} = 0, \ \ \
{\lambda}^{\alpha}\!\,_{||\mu} = 0.
\end{equation}
 In analogy to the classical AP-space \cite{AMR}, the above relations are referred to as the generalized AP-condition.\\
 \par

The canonical connection can be explicitly expressed in terms of the ${\lambda}$'s in the form
\begin{equation}\label{canonical GAP}
\Gamma^{\alpha}_{\mu\nu} = \undersym{\lambda}{i}^{\alpha}
(\delta_{\nu} \ \undersym{\lambda}{i}_{\mu}), \ \ \
{C}^{\alpha}_{\mu\nu} =
\undersym{\lambda}{i}^{\alpha}(\dot{\partial}_{\nu} \
\undersym{\lambda}{i}_{\mu}).
\end{equation}
\par
We now list some important tensors which prove useful later on.
Let
\begin{equation} T^{\alpha}_{\mu\nu} :=
{\Gamma}^{\alpha}_{\mu\nu} - {\Gamma}^{\alpha}_{\nu\mu}
, \ \ \ \ S^{\alpha}_{\mu\nu}: = {C}^{\alpha}_{\mu\nu} - {C}^{\alpha}_{\nu\mu}.
\end{equation}
Then $T^{\alpha}_{\mu\nu}$ and $S^{\alpha}_{\mu\nu}$ are referred to as the $(h)h$- and $(v)v$-torsion tensor of the canonical connection.\\
\par
Let
\begin{equation}\label{contortion} A^{\alpha}_{\mu\nu}: = \
\undersym{\lambda}{i}^{\alpha} \ \undersym{\lambda}{i}_{\mu{o \atop
|}\nu} = {\Gamma}^{\alpha}_{\mu\nu} -  \
\Ahmed{\Gamma}{\circ}^{\alpha}_{\mu\nu}, \ \ \ \ B^{\alpha}_{\mu\nu}: =
\ \undersym{\lambda}{i}^{\alpha} \ \undersym{\lambda}{i}_{\mu{o\atop
||}\nu} = {C}^{\alpha}_{\mu\nu} -
\Ahmed{C}{\circ}^{\alpha}_{\mu\nu}.
\end{equation}
We refer to $A^{\alpha}_{\mu\nu}$
and $B^{\alpha}_{\mu\nu}$ as the $h$- and $v$-contortion tensors,
respectively.\\
\par
We have the following important relations between the torsion and the contortion tensors:
\begin{equation} T^{\alpha}_{\mu\nu} = A^{\alpha}_{\mu\nu} - A^{\alpha}_{\nu\mu}, \ \ \ \ S^{\alpha}_{\mu\nu}=
B^{\alpha}_{\mu\nu} - B^{\alpha}_{\nu\mu}.
\end{equation}

\par
In addition to the Miron and the canonical connections, our space admits other natural
Finsler connection. In analogy to the classical AP-space \cite{AMR}, we define the
\textbf{dual} Finsler connection \\ $\widetilde{D} = (\widetilde{\Gamma}^{\alpha}_{\mu\nu},\widetilde{N}^{\alpha}_{\mu},\widetilde{C}^{\alpha}_{\mu\nu})$ by
\begin{equation}\label{dual GAP}
\widetilde{\Gamma}^{\alpha}_{\mu\nu} :=
 \Gamma^{\alpha}_{\nu\mu}, \ \ \widetilde{N}^{\alpha}_{\mu} = N^{\alpha}_{\mu} ,  \ \
\widetilde{C}^{\alpha}_{\mu\nu} := C^{\alpha}_{\nu\mu}
\end{equation}
The $h$- and $v$-covariant derivatives with
respect to the dual connection $\widetilde{D}$ will be denoted by
\lq\lq\,$\widetilde{|}\,"$ and \lq\lq\,$\widetilde{||}\,"$.
\par
Summing up, in a GAP-space we have at least three naturel Finsler connections, namely the Miron, canonical and dual connections. The first two connections are metric while the third is not. The first one is symmetric while the others are not. The second one is flat while the others are not. The nonlinear connection $N^{\alpha}_{\beta}$ of these three Finsler connections is the same. For a fixed nonlinear connection, we refer to it as the nonlinear connection of the GAP-space.

\Section{Finslerized Parallelizable space}
\hspace{12pt}
In this section, we construct a Finsler structure on a GAP-space. We refer to this new space as a Finslerized Parallelizable  space or simply an FP-space. An FP-space combines within its geometric structure the richness of Finsler geometry and the richness and simplicity of the GAP-geometry. Moreover, in an FP-space both   geometric structures are closely interrelated and intertwined. This follows from the fact that the two metrics of the two geometric structures coincide (Theorem $3.2$).\\
 \par
 The next theorem plays the key role in the construction of a Finsler structure on a GAP-space.
\begin{theorem}
Let $(M, \,\, \undersym{\lambda}{i}(x,\,y))$ be a GAP-space.
\begin{description}
  \item[(a)] If $g_{\mu \nu}(x,\,y) := \undersym{\lambda}{i}_{\mu}(x,\,y) \,\, \undersym{\lambda}{i}_{\nu}(x,\,y)$,
  then $(M,\,g)$ is a generalized Lagrange space.
  \item[(b)] If $L(x,\,y) := g_{\mu \nu}(x,\,y) \, y^{\mu}\, y^{\nu}$
  such that $y^{\mu}\, \dot{\partial}_{\alpha} g_{\mu \nu} = 0,$
  then $(M,\,L(x,\,y))$ is \;a Lagrange space.
  \item[(c)] If $ {F}(x,\,y) := \sqrt{g_{\mu \nu}(x,\,y)\, y^{\mu}\, y^{\nu}}$ \footnote{We note that $g_{\mu \nu}(x,\,y)\, y^{\mu}\, y^{\nu}=\,  \undersym{\lambda}{i}_{\mu} \,\, \undersym{\lambda}{i}_{\nu}\, y^{\mu}\, y^{\nu} = \,(\undersym{\lambda}{i}_{\mu} \, y^{\mu})\, ( \undersym{\lambda}{i}_{\nu}\,y^{\nu}) = (\,\undersym{\lambda}{i}_{\mu} \, y^{\mu})^{2} \geq 0 $.} such that
        \begin{enumerate}
      \item The tensor $\dot{\partial}_{\alpha} g_{\mu \nu} \text{ is totally symmetric},$
       \item The parallelization vector fields $\,\undersym{\lambda}{i}(x,\,y)$ are positively homogenous of degree zero in y,
\end{enumerate}
\end{description}
\hspace{21pt}then $(M,\, F(x,\,y))$ is a Finsler space.
\footnote{A result similar to \textbf{(c)} is found in \cite{Wanas 2009}, but in a completely different context and under more restrictive conditions.}
\end{theorem}
\noindent In this case, we refer to $(M,\,\, \undersym{\lambda}{i}(x,\,y),\, {F}(x,\,y))$ as a Finslerized Absolute Parallelism space or a Finslerized Parallelizable  space (FP-space).
\begin{proof}
\begin{description}
 \item[]
  \item[(a)] See Theorem $2.2$.
  \item[(b)] \label{Lagrange} By hypothesis, we have $y^\beta \dot{\partial}_{\nu} g_{\alpha \beta} = 0$. Consequently, $\dot{\partial}_{\nu} L = (\dot{\partial}_{\nu} y^{\alpha})\, y^\beta \, g_{\alpha \beta}+ y^\alpha \, (\dot{\partial}_{\nu} y^{\beta})\, g_{\alpha \beta}+  y^{\alpha}\, y^\beta \,\dot{\partial}_{\nu} g_{\alpha \beta} = \delta^\alpha_\nu \, y^\beta \, g_{\alpha \beta}+ y^\alpha \, \delta^\beta_\nu \, g_{\alpha \beta} = 2 \, y^\alpha \, g_{\nu \alpha}.$
 Again, using $y^\beta \dot{\partial}_{\nu} g_{\alpha \beta} = 0$, we conclude that
   $$ \frac{1}{2}\dot{\partial}_{\mu} \dot{\partial}_{\nu} L = (\dot{\partial}_{\mu} y^{\alpha})\, g_{\nu \alpha} + y^{\alpha}\,(\dot{\partial}_{\mu}g_{\nu \alpha}) = \delta^\alpha_\mu \, g_{\nu \alpha}= g_{\mu \nu}.$$
  \item[(c)]  We show that the four conditions defining a Finsler structure are satisfied:
It is clear that $F$ is $C^{\infty}$ on {\fontsize{18}{10}{$\tau$}}$M$ and $C^{0}$ on $TM$. On the other hand, $ {F}(x,\,y) \geq 0$ as the square root always give non-negative values and $F(x,0) =0$.
Moreover, if $\alpha$ is a positive number then, noting that $g_{\mu \nu}$ is positively homogenous of degree zero
   in $y$, $F(x, \alpha y)= \sqrt{g_{\mu \nu}(x,\alpha y)\, (\alpha y)^{\mu}\, (\alpha y)^{\nu}} =  \sqrt{ \alpha^{0}\, g_{\mu \nu}(x, y) \alpha^{2}\, y^{\mu}\, y^{\nu}} = \alpha \sqrt{ g_{\mu \nu}(x, y) \, y^{\mu} \,y^{\nu}} = \alpha\, {F}(x,\,y).$ Finally, the total symmetry of the tensor  $\dot{\partial}_{\mu} g_{\alpha \beta}$ implies that $y^{\alpha}\, \dot{\partial}_{\mu} g_{\alpha \beta}= y^{\alpha}\, \dot{\partial}_{\alpha} g_{\mu \beta} = 0$, by Euler theorem on homogenous functions.
Thus, similar to the proof of $(b)$, we get $$\frac{1}{2}
\dot{\partial_{\mu}} {\dot{\partial_{\nu}}} F^{2} = g_{\mu \nu}$$
\end{description}
\vspace{-0.8cm}
\end{proof}
As a consequence of the above theorem, we have the following simple and important result.
\begin{theorem}
The GAP-metric and the Finsler metric coincide: $$g_{\mu \nu} = \undersym{\lambda}{i}_{\mu}(x,\,y) \,\, \undersym{\lambda}{i}_{\nu}(x,\,y) =\frac{1}{2} \dot{\partial_{\mu}} {\dot{\partial_{\nu}}} F^{2}$$
\end{theorem}
\begin{corollary}Let $(M,\, \, \undersym{\lambda}{i}(x,\,y),\,{F}(x,\,y))$ be an FP-space.
Then the following hold:
\begin{description}
\item[(a)] Both $\dot{\partial}$-Christoffel symbols of the GAP-space and the Finsler space coincide.
\item[(b)]Both formal Christoffel symbols of the GAP-space and the Finsler space coincide.
\item[(c)]The nonlinear connection of the GAP-space coincides with the Barthel connection.
  \item[(d)]
  Both the $\delta$-Christoffel symbols of the GAP-space and the Finsler space coincide. 
\end{description}
\end{corollary}
Consequently, the nonlinear connection of all Finsler connections of the FP-space is the Barthel connection. It can be written in terms of the building blocks $\lambda(x,y)$'s of the FP-space only as follows
$$N^{\alpha}_{\mu}=y^{\beta}\, \Ahmed{\Gamma}{\circ}\!^{\alpha}_{\beta\mu}= y^{\beta}\, (\Gamma^{\alpha}_{\beta\mu} -A^{\alpha}_{\beta\mu})= y^{\beta}\, \undersym{\lambda}{i}^{\alpha}\,(\delta_{\mu} \,\undersym{\lambda}{i}_{\beta}- \,\undersym{\lambda}{i}_{\beta {o\atop |} \mu}).$$
Therefore, we have the following results.
\begin{corollary}
The $(v)h$-torsion tensor $R^\alpha_{\mu \nu}$ of all connections, defined on an FP-space, is the same and is given by
$$R^\alpha_{\mu \nu} = \delta_{\nu} N^\alpha_{\mu}- \delta_{\mu} N^\alpha_{\nu}.$$
\end{corollary}
\begin{theorem}
Let $(M,\,\,\undersym{\lambda}{i}(x,\,y),\,{F}(x,\,y))$ be an FP-space. Then the Cartan connection coincides with the Miron connection.
\end{theorem}
 In an FP-space, there are both Finsler and GAP-connections. Our aim now is to express all connections defined on this space in terms of one of them.
  We here express all the FP-connections in terms of the canonical connection only, which is, in turn, expressed in terms of the vector fields forming the parallelization $\,\undersym{\lambda}{i}(x,y)$ (cf. (2.9)).
\begin{theorem}
The Cartan (Miron), Berwald and dual connections can be expressed in terms of the canonical connection $D$ as follows:
\begin{description}
  \item[(a)] The Cartan connection $C\Gamma = (\Ahmed{\Gamma}{\circ}\!^{\alpha}_{\mu\nu},  N^{\alpha}_{\mu}, \Ahmed{C}{\circ}\!^{\alpha}_{\mu\nu})$ is given by
  \begin{equation}\label{Cartan}
  \Ahmed{\Gamma}{\circ}\!^{\alpha}_{\mu\nu} = {\Gamma}^{\alpha}_{\mu\nu} - A^{\alpha}_{\mu\nu}, \ \ \ \ \Ahmed{C}{\circ}\!^{\alpha}_{\mu\nu} = {C}^{\alpha}_{\mu\nu} - B^{\alpha}_{\mu\nu}.
  \end{equation}
\item [(b)]The Berwald connection $B \Gamma = (G^{\alpha}_{\mu\nu}, N^{\alpha}_{\beta}, 0)$ is given by
      \begin{equation}\label{Berwald}
        G^{\alpha}_{\mu\nu} ={\Gamma}^{\alpha}_{\mu\nu} +{P}^{\alpha}_{\mu\nu}.\end{equation}
         \item [(c)]
         The dual connection $\widetilde{D}= (\widetilde{\Gamma}^{\alpha}_{\mu\nu} ,N^{\alpha}_{\beta} ,\widetilde{C}^{\alpha}_{\mu\nu})$ is given by
  \begin{equation}
\widetilde{\Gamma}^{\alpha}_{\mu\nu} = \Gamma^{\alpha}_{\nu\mu}, \ \ \ \
\widetilde{C}^{\alpha}_{\mu\nu} = C^{\alpha}_{\nu\mu}.
  \end{equation}
\end{description}
\end{theorem}
In fact, \textbf{(a)} follows from (2.11), \textbf{(b)} follows from definition 1.7 \textbf{(d)}, and \textbf{(c)} follows from (2.13).



\begin{proposition}
The $(h)h$-, $(h)hv$-, $(v)hv$- and $(v)v$-torsion tensors of the canonical connection
${D}$ are  given respectively by:
\begin{description}
  \item[(a)] ${T}^\alpha_{\mu \nu}={\Gamma}^\alpha_{\mu \nu}- {\Gamma}^\alpha_{\nu \mu},$
  \item[(b)] ${C}^\alpha_{\mu \nu}:= \undersym{\lambda}{i}^{\alpha}(\dot{\partial}_{\nu}\, \undersym{\lambda}{i}_{\mu})$ the connection parameters,
  \item[(c)] ${P}^\alpha_{\mu \nu}= y^{\beta} \, [{C}^{\alpha}_{\mu \nu | \beta} - B^{\alpha}_{\mu \nu | \beta} - A^{\alpha}_{\eta \beta} ({C}^{\eta}_{\mu \nu}\,+ B^{\eta}_{\mu \nu}) +  A^{\eta}_{\mu \beta}(C^{\alpha}_{\eta \nu}\,+ B^{\alpha}_{\eta \nu}) +  A^{\eta}_{\nu \beta}({C}^{\alpha}_{\mu \eta} + B^{\alpha}_{\mu \eta}\,)] - A^{\alpha}_{\mu\nu},$
  \item[(d)] ${S}^\alpha_{\mu \nu}={C}^\alpha_{\mu \nu}-{C}^\alpha_{\nu \mu}.$
\end{description}
\end{proposition}
We note that in the proof of item $\textbf{(c)}$ above, we have used the relation:
\begin{equation}\label{relation bet. canonical and cartan}
X^{\alpha}_{\mu \nu{o \atop
|} \beta} = {X^{\alpha}_{\mu \nu}}_{ | \beta} - X^{\eta}_{\mu \nu}\, A^{\alpha}_{\eta \beta} + X^{\alpha}_{\eta \nu} \, A^{\eta}_{\mu \beta} +  X^{\alpha}_{\mu \eta} \, A^{\eta}_{\nu \beta} \end{equation}
\par
     We have the following crucial result.
\begin{theorem}\cite{GAP} The $h$-, $v$- and $hv$-curvature tensors of the canonical connection
vanish identically.
\end{theorem}
 The above theorem, namely, the vanishing of the curvature tensors of the canonical connection, allows us to express all curvature tensors of the FP-space in terms of the torsion tensors of the canonical connection only as will be shown in the following three propositions. The proofs of the first two propositions are similar to the proofs of Proposition $6.1$ and Proposition $6.3$ of \cite{GAP} \footnote{It should be noted that the $hv$-curvature tensor of the dual connection in Proposition $3.9$ differs by a sign from the $hv$-curvature tensor of the dual connection in Proposition $6.1$ of \cite{GAP}}. Proposition $3.11$ can be proved in a similar fashion.

\begin{proposition}
The torsion and curvature tensors of the dual connection $\widetilde{D} $ can be expressed in the form:
\begin{description}
\item[(a)] The $(h)h$-, $(h)hv$-, $(v)hv$- and $(v)v$-torsion tensors of $\widetilde{D}$ are given by:\vspace{5pt}\\
\begin{minipage}{0.30 \textwidth}
\begin{description}
\item[(1)] $\widetilde{T}^{\alpha}_{\mu \nu}= - T^{\alpha}_{\mu\nu},$
\item[(3)] $\widetilde{P}^{\alpha}_{\mu \nu} = {P}^{\alpha}_{\nu \mu},$
\end{description}
\end{minipage}
\begin{minipage}{0.49 \textwidth}
\vspace*{-.2cm}
\begin{description}
\item[(2)] $\widetilde{C}^{\alpha}_{\mu \nu} := {C}^{\alpha}_{\nu\mu},$
\item[(4)]$\widetilde{S}^{\alpha}_{\mu \nu} = -S^{\alpha}_{\mu\nu}.$
\end{description}
\end{minipage}

\item[(b)] The $h$-, $v$- and $hv$-curvature tensors of $\widetilde{D}$ are given by:
\begin{description}
\item[(1)] $\widetilde{R}^{\alpha}_{\mu\sigma\nu} =
{T}^{\alpha}_{\sigma\nu |\mu} + {C}^{\alpha}_{\eta\mu}
{R}^{\eta}_{\sigma\nu} + {C}^{\alpha}_{\sigma \eta} {R}^{\eta}_{\nu\mu} +
{C}^{\alpha}_{\nu \eta }{R}^{\eta}_{\mu\sigma},$
\item[(2)] $\widetilde{S}^{\alpha}_{\mu\sigma\nu} = {S}^{\alpha}_{\sigma\nu ||\mu}$,
\item[(3)] $\widetilde{P}^{\alpha}_{\nu\mu\sigma} = {S}^{\alpha}_{\nu\mu |\sigma} +
{T}^{\alpha}_{\sigma\nu ||\mu} - {S}^{\eta}_{\mu\nu}
{T}^{\alpha}_{\sigma\eta} + {S}^{\alpha}_{\mu\eta}
{T}^{\eta}_{\sigma\nu} + {T}^{\alpha}_{\eta\nu} \, {C}^{\eta}_{\sigma\mu} +
{P}^{\eta}_{\sigma\mu} \, {S}^{\alpha}_{\eta\nu}.$
\end{description}
\end{description}
\end{proposition}

\begin{proposition}
The torsion and curvature tensors of the Cartan (Miron) connection $C \Gamma$ can be expressed in the form:
\begin{description}
\item[(a)] The $(h)h$- and $(v)v$-torsion tensors of $C \Gamma$ vanish and the  $(h)hv$-torsion tensors of $C \Gamma$ is given as in $(\ref{Cartan})$ and $(v)hv$-torsion tensors of $C \Gamma$ is given by:
\begin{equation}\label{p-torsion cartan}   \overcirc{P}\!^{\alpha}_{\mu \nu} ={P}^{\alpha}_{\mu \nu} + A^{\alpha}_{\mu \nu}  \end{equation}
\item[(b)] The $h$-, $hv$- and $v$-curvature tensors of $C \Gamma$ are given by:
    \begin{description}
\item[(1)]  $\overcirc{R}\!^{\alpha}_{\mu \nu \sigma} = \mathfrak{U}_{(\sigma,\nu)} \{A^{\alpha}_{\mu \sigma | \nu} + A^{\eta}_{\mu \nu} \, A^{\alpha}_{\eta\sigma}\} +A^{\alpha}_{\mu \eta} {T}^{\eta}_{\sigma \nu} + B^{\alpha}_{\mu \eta} {R}^{\eta}_{\sigma\nu},$
   \item[(2)] $\overcirc{P}\!^{\alpha}_{\mu \nu \sigma} =  B^{\alpha}_{\mu \sigma |\nu }- A^{\alpha}_{\mu \nu ||\sigma } + {T}^{\eta}_{\sigma \nu} ({C}^{\alpha}_{\mu \eta} - B^{\alpha}_{\mu \eta}) - A^{\alpha}_{\mu\eta}\, {C}^{\eta}_{\nu \sigma}- B^{\eta}_{\mu\sigma} \, A^{\alpha}_{\eta\nu }+ B^{\alpha}_{\eta\sigma} \, A^{\eta}_{\mu\nu}- B^{\alpha}_{\mu\eta} \, {P}^{\eta}_{\nu \sigma}, $
   \item[(3)] $\overcirc{S}\!^{\alpha}_{\mu \nu \sigma} = \mathfrak{U}_{(\nu,\sigma)} \{B^{\alpha}_{\mu \nu || \sigma} + B^{\eta}_{\mu \sigma} \, B^{\alpha}_{\eta\nu}\} +B^{\alpha}_{\mu \eta} {S}^{\eta}_{\nu \sigma}$.

 \end{description}
\end{description}
\end{proposition}

\begin{proposition}The torsion and curvature tensors of the Berwald connection $B\Gamma$ can be expressed as follows:
\begin{description}
\item[(a)]
 The $(h)h$-, $(h)hv$-, $(v)hv$- and $(v)v$-torsion tensors of $B\Gamma$ vanish.
\item[(b)] The $h$-, $v$- and $hv$-curvature tensors of the $B\Gamma$  are given by:
\begin{description}
\item[(1)] $\bar{R}^{\alpha}_{\mu\nu \sigma} = \mathfrak{U}_{(\nu,\sigma)} \{{P}^{\alpha}_{\mu \nu | \sigma} + {P}^{\eta}_{\mu \nu} \, {P}^{\alpha}_{\eta\sigma}\} +{P}^{\alpha}_{\mu \eta} {T}^{\eta}_{\nu \sigma} - {C}^{\alpha}_{\mu \eta} {R}^{\eta}_{\nu\sigma},$
\item[(2)] $\bar{S}^{\alpha}_{\mu\sigma\nu} = 0.$
\item[(3)] $\bar{P}^{\alpha}_{\nu\mu\sigma} = \dot{\partial}_{\sigma}{P}^{\alpha}_{\nu\mu}+ \dot{\partial}_{\sigma}\,{\Gamma}^{\alpha}_{\nu\mu}.$
\end{description}
\end{description}
\end{proposition}
\begin{remark}
\em{The above expressions show that the curvature tensors of all connections existing in the FP-space are formulated in terms of the torsion and contortion  tensors of the canonical connection. However, there is a strong relation between the torsion and contortion tensors as shown in the following relations \cite{GAP}:
\begin{equation}
A_{\mu\nu\sigma} = \frac{1}{2}(T_{\mu\nu\sigma}+T_{\sigma\nu\mu}+T_{\nu\sigma\mu})
\end{equation}
and
\begin{equation}
B_{\mu\nu\sigma} = \frac{1}{2}(s_{\mu\nu\sigma}+s_{\sigma\nu\mu}+s_{\nu\sigma\mu})
\end{equation}
where \,$T_{\mu\nu\sigma}:= g_{\eta\mu}\,T^{\eta}_{\nu\sigma}, \, \, s_{\mu\nu\sigma}:= g_{\eta\mu}\, S^{\eta}_{\nu\sigma}, \,\, A_{\mu\nu\sigma}:= g_{\eta\mu}\,A^{\eta}_{\nu\sigma}$ \,and \,$B_{\mu\nu\sigma}:= g_{\eta\mu}\,B^{\eta}_{\nu\sigma}$.}\end{remark}
This means that all curvature tensors can be expressed solely in terms of the torsion tensors of the canonical connection.
\par
The next table provides a comparison concerning the different connections of an FP-space as well as their associated torsion and curvature tensors.
\vspace{8pt}
\begin{center}{{Table 1: Summary of the connections of an FP-space}}
\\[0.3 cm]

\small{\begin{tabular} {|c|c|c|c|c|}\hline
 \multirow{2}{*}{\textbf{} }&\multirow{2}{*}{Cartan (Miron)}&
 \multirow{2}{*}{Canonical} &\multirow{2}{*}{Dual }&\multirow{2}{*}{ Berwald}\\
 &&&&\\[0.1 cm]\hline
\multirow{2}{*}{$({F}^{\alpha}_{\mu\nu},
{N}^{\alpha}_{\mu},{C}^{\alpha}_{\mu\nu})$}&
\multirow{2}{*}{$(\,\overcirc{\Gamma}\!^{\alpha}_{\mu\nu}, N^{\alpha}_{\mu},\,\overcirc{C}\!^{\alpha}_{\mu\nu})$}&
\multirow{2}{*}{$( \Gamma^{\alpha}_{\mu\nu}, N^{\alpha}_{\mu}, C^{\alpha}_{\mu\nu})$}&
\multirow{2}{*}{$( \widetilde{\Gamma}^{\alpha}_{\mu\nu}, N^{\alpha}_{\mu}, \widetilde{C}^{\alpha}_{\mu\nu})$}&
\multirow{2}{*}{$( G^{\alpha}_{\mu\nu}, N^{\alpha}_{\mu}, 0)$}\\
&&&&
\\[0.1 cm]\hline
\multirow{2}{*}{{(h)h-torsion} ${T}^{\alpha}_{\mu\nu}$}&\multirow{2}{*}{$0$}&\multirow{2}{*}{${T}^{\alpha}_{\mu\nu}$}&
\multirow{2}{*}{$\widetilde{T}^{\alpha}_{\mu\nu}= -T^{\alpha}_{\mu\nu}$}&\multirow{2}{*}{$0$}\\[0.2 cm]
\multirow{2}{*}{{(h)hv-torsion}
${C}^{\alpha}_{\mu\nu}$}&\multirow{2}{*}{$\overcirc{C}\!^{\alpha}_{\mu\nu}= C^{\alpha}_{\mu\nu}-B^{\alpha}_{\mu\nu}$} &
\multirow{2}{*}{$C^{\alpha}_{\mu\nu}$}&
\multirow{2}{*}{$\widetilde{C}^{\alpha}_{\mu\nu}={C}^{\alpha}_{\nu\mu}$}&\multirow{2}{*}{$0$}
\\[0.3 cm]
{ (v)h-torsion} ${R}^{\alpha}_{\mu\nu}$&$\overcirc{R}\!^{\alpha}_{\mu\nu}={R}^{\alpha}_{\mu\nu}$
&$R^{\alpha}_{\mu\nu}$&
$\widetilde{R}^{\alpha}_{\mu\nu}=R^{\alpha}_{\mu\nu}$&
$\bar{R}^{\alpha}_{\mu\nu}={R}^{\alpha}_{\mu\nu}$\\[0.1 cm]
{ (v)hv-torsion}
${P}^{\alpha}_{\mu\nu}$&$\overcirc{P}\!^{\alpha}_{\mu\nu}={P}^{\alpha}_{\mu\nu}+ A^{\alpha}_{\mu\nu}$&
  ${P}^{\alpha}_{\mu\nu}$&$\widetilde{P}^{\alpha}_{\mu\nu}={P}^{\alpha}_{\nu\mu}$ &$0$
\\[0.1 cm]
{ (v)v-torsion} ${S}^{\alpha}_{\mu\nu}$&$0$ &${S}^{\alpha}_{\mu\nu}$&$\widetilde{S}^{\alpha}_{\mu\nu}= -S^{\alpha}_{\mu\nu}$ &$0$
\\[0.2 cm]\hline
{ h-curvature} ${R}^{\alpha}_{\mu\nu\sigma}$& $\overcirc{R}\!^{\alpha}_{\mu\nu\sigma}$&
$0$&
$\widetilde{R}^{\alpha}_{\mu\nu\sigma}$&$\bar{R}^{\alpha}_{\mu\nu\sigma}$
\\[0.1 cm]
{hv-curvature} ${P}^{\alpha}_{\mu\nu\sigma}$& $\overcirc{P}\!^{\alpha}_{\mu\nu\sigma}$&
$0$&
$\widetilde{P}^{\alpha}_{\mu\nu\sigma}$&$\bar{P}^{\alpha}_{\mu\nu\sigma}$
\\[0.1 cm]

{v-curvature}
${S}^{\alpha}_{\mu\nu\sigma}$& ${S}^{\alpha}_{\mu\nu\sigma}$&
$0$&$\widetilde{S}^{\alpha}_{\mu\nu\sigma}$ &$0$
\\[0.1 cm]\hline 
\end{tabular}}
\end{center}


\Section{Special Finslerized Parallelizable spaces}
\hspace{12pt}
In this section we study  some special cases of an FP-space. Further conditions are assumed. In both the FP-Landsberg and FP-Minkowskian spaces, the extra conditions are imposed on the Finsler structure only whereas in the FP-Riemannian space the extra conditions are imposed on the GAP-structure only. Finally, in the FP-Berwald space the extra conditions are imposed on both the Finsler and the GAP-structures. The definitions of special Finsler spaces treated here may be found in \cite{Mats} and \cite{Amr soliman}.

\subsection{FP-Landsberg and FP-Berwald spaces}
\begin{definition}
A Landsberg space is a Finsler space satisfying the condition $\overcirc{P}\!^{\alpha}_{\mu \nu \sigma} \, y^{\mu}=0$.\end{definition}

\begin{definition}
A Finsler space (M,\,F) is said to be a Berwald space if the $hv$-Cartan curvature tensor vanishes i.e. $\Ahmed{P}{\circ}^{\alpha}_{\mu\nu \sigma}=0$.\end{definition}
\par
We introduce the following definitions.
\begin{definition}
An FP-Landsberg space is an FP-space equipped with the condition $\overcirc{P}\!^{\alpha}_{\mu \nu \sigma} \, y^{\mu}=0$.
\end{definition}
\begin{definition}
An FP-Berwald space is an FP-space for which $\Ahmed{P}{\circ}^{\alpha}_{\mu\nu \sigma}=0$ and ${\Gamma}^{\alpha}_{\mu \nu}(x, y) \equiv {\Gamma}^{\alpha}_{\mu \nu}(x)$.\end{definition}
\begin{remark}
\em{ It should be noted that in an FP-Berwald manifold the Cartan tensor $\Ahmed{C}{\circ}\!\!_{\alpha\mu\nu}$ and the vertical counterpart ${C}^{\alpha}_{\mu\nu}$ of the canonical connection remain unaltered. Hence, the $v$-contortion tensor and the $v$-curvature tensors are the same as in the general case.}
\end{remark}
\begin{proposition}
Let $(M,\,\undersym{\lambda}{i}(x,\,y),F(x,\,y))$ be an FP-Berwald manifold.
\begin{description}
\item[(a)] The horizontal counterpart of the Cartan and Berwald connections coincide.
\item[(b)]The horizontal contortion tensor $A^{\alpha}_{\mu\nu}$  and the horizontal counterpart of the canonical and dual connections are functions of the positional argument only.
\end{description}
\end{proposition}
\begin{proof}
\begin{description}
  \item[(a)]It follows from \cite{LDG} that
 $$ \Ahmed{P}{\circ}^{\alpha}_{\mu\nu \sigma}=0 \Longleftrightarrow \,\overcirc{\Gamma}\!^{\alpha}_{\mu \nu}(x, y) \equiv \,\overcirc{\Gamma}\!^{\alpha}_{\mu \nu}(x) \Longleftrightarrow \, \Ahmed{C}{\circ}\!^{\alpha}_{\mu \nu {o \atop
|} \beta} =0 \Longrightarrow \overcirc{P}\!^{\alpha}_{\mu\nu}=0 \Longrightarrow \, \overcirc{\Gamma}\!^{\alpha}_{\mu \nu}(x) = G^{\alpha}_{\mu \nu}(x).$$

  \item[(b)] It is a direct outcome of the first part of equation $(3.1)$, the definition of the FP-Berwald manifold and equation $(3.3)$.
\end{description}
\end{proof}
Every FP-Berwald space is an FP-Landsberg space. The converse is not true in general. Nevertheless, we have the following partial converse.
\begin{theorem}
\begin{description}
\item[]
  \item[(a)] Assume that both ${\Gamma}^{\alpha}_{\mu\nu}$ and  $A^{\alpha}_{\mu\nu}$ are functions of $x$ only. Then an FP-Landsberg space is an FP-Berwald space.
  \item[(b)] Assume that both ${\Gamma}^{\alpha}_{\mu\nu}$ and ${P}^{\alpha}_{\mu\nu}$ are functions of $x$ only. Then an FP-Landsberg space is an FP-Berwald space.
\end{description}
\end{theorem}
\begin{proof}
Clearly this holds from the first part of $(\ref{Cartan})$ and from $(\ref{Berwald})$, respectively.
\end{proof}

The next table provides a comparison concerning the connections of an FP-Berwald space and their associated torsion and curvature tensors.

\vspace{8pt}
{\begin{center}{{Table 2: Torsion and curvature tensors of an FP-Berwald space}}
\\[0.4 cm]
{
\small\fontsize{12}{15}\selectfont{\begin{tabular} {|l|c|c|c|c|c|}\hline
 &{{\fontsize{12}{12}\selectfont Cartan }}&
 { {\fontsize{12}{12}\selectfont Canonical }} &{{\fontsize{12}{12}\selectfont Dual }} &{{\fontsize{12}{12}\selectfont Berwald}}
\\[0.1 cm]\hline
\multirow{2}{*}{{ (h)h-torsion} $T^{\alpha}_{\mu\nu}$}&\multirow{2}{*}{$0$}&\multirow{2}{*}{${T}^{\alpha}_{\mu\nu}$}&\multirow{2}{*}{$ -{T}^{\alpha}_{\mu\nu}$}&\multirow{2}{*}{$0$}\\[0.1cm]
\multirow{2}{*}{{(h)hv-torsion}
$ C^{\alpha}_{\mu\nu}$}&\multirow{2}{*}{${C}^{\alpha}_{\mu\nu} - B^{\alpha}_{\mu \nu}$} &
\multirow{2}{*}{${C}^{\alpha}_{\mu\nu}$}&
\multirow{2}{*}{${C}^{\alpha}_{\nu\mu}$}&
\multirow{2}{*}{$0$}
\\[0.3 cm]
{(v)h-torsion} $R^{\alpha}_{\mu\nu}$&${R}^{\alpha}_{\mu\nu}$
&${R}^{\alpha}_{\mu\nu}$&
${R}^{\alpha}_{\mu\nu}$&
${R}^{\alpha}_{\mu\nu}$\\[0.1 cm]
{(v)hv-torsion}
$P^{\alpha}_{\mu\nu}$&$0$&
  $-A^{\alpha}_{\mu\nu}$&$-A^{\alpha}_{\nu\mu}$ &$0$
  \\[0.1 cm]
{(v)v-torsion} $S^{\alpha}_{\mu\nu}$&$0$ &${S}^{\alpha}_{\mu\nu}$&$-{S}^{\alpha}_{\mu\nu}$ &$0$
\\[0.1 cm]\hline
{ h-curvature} $R^{\alpha}_{\mu\nu \sigma}$& $\overcirc{R}\!^{\alpha}_{\mu\nu \sigma}$&
$0$&
$\widetilde{R}^{\alpha}_{\mu\nu \sigma}$&$\bar{R}^{\alpha}_{\mu\nu \sigma}$
\\[0.1 cm]
{hv-curvature} $P^{\alpha}_{\mu\nu \sigma}$&
$0$&$0$&$\widetilde{P}^{\alpha}_{\mu\nu \sigma}$&$0$
\\[0.1 cm]
{v-curvature}
$S^\alpha_{\mu\nu\sigma}$&$\overcirc{S}\!^\alpha_{\mu\nu\sigma}$&
$0$&$\widetilde{S}^{\alpha}_{\mu\nu \sigma}$ &$0$
\\[0.1 cm]\hline
\end{tabular}}
}
\end{center}}
\vspace{0.1cm}
\par
The form of the curvature tensors is similar to the general case with some geometric objects, namely, ${T}^{\alpha}_{\mu\nu}$ and ${P}^{\alpha}_{\mu\nu}$, loosing their dependence on $y$.

\subsection{FP-Minkowskian space}
\begin{definition}
A Finsler manifold (M,\,F) is said to be locally Minkowskian if the $h$- and $hv$- Cartan curvature tensors both vanish: $\overcirc{P}\!^{\alpha}_{\mu\nu \sigma}=0 =\, \overcirc{R}\!^{\alpha}_{\mu\nu \sigma}$.\end{definition}

We have the following useful characterization of locally Minkowskian spaces \cite{LDG}.
A Finsler manifold $(M, F)$ is locally Minkowskian if, and only if, at each $p \in M$ there exists a local coordinate system $(x^i, y^i)$ on $TM$ such that $F$ is a function of $y^i$ only. We refer to this local coordinate system as the \textbf{natural} coordinate system. It can be shown that $F$ is a function of $y^i$ only if, and only if, the Finsler metric $g_{\mu\nu}$ is a function of $y^i$ only. This follows from the fact that $F^{2}$ is expressed in the form $F^{2}= g_{\mu\nu}\, y^{\mu}\,y^{\nu}$. This implies, in particular, that $\,\overcirc{C}\!^{\alpha}_{\mu \nu} $ are functions of $y^{i}$ only. Moreover, as can be checked \cite{FB},
$$ g_{\mu \nu}(x, y) \equiv g_{\mu \nu}(y) \text{ in some coordinate system }
 \Longleftrightarrow \, \overcirc{C}\!_{\mu \nu {o \atop
|} \beta}^{\alpha} =0 \text{ and }  \overcirc{R}\!^{\alpha}_{\mu \nu \beta} =0.$$
\begin{remark}
\em{ The dependence of a tensor (or a linear connection) on $y$ only is not coordinate independent. If a tensor (or a linear connection) is a function of $y$ only in some coordinate system then this does not necessarily imply that it remains a function of $y$ in other coordinate systems. This is because the transformation formula involves $p^{\alpha'}_{\alpha}$ which are functions of $x$ only. On the other hand, for the same reason, the dependence on $x$ is coordinate independent.}
\end{remark}
\begin{theorem}
Let $(M, \,\undersym{\lambda}{i}(x,\,y),\,F(x,\,y))$ be an FP-Minkowskian manifold. Then, in the natural coordinate system, the following hold:
\begin{description}
  \item[(a)] The Barthel connection $N^{\alpha}_{\beta}$ vanishes.
      \item[(b)] The $\delta$-Christoffel symbols $\overcirc{\Gamma}\!^{\alpha}_{\mu\nu}$ vanish.
\item[(c)] The Berwald connection vanishes whereas the Cartan connection reduces to $(0, 0, \, \overcirc{C}\!^{\alpha}_{\mu\nu})$.
  \item[(d)] The horizontal covariant derivatives with respect to the Cartan and  Berwald connections reduce to  partial differentiation.
  \item[(e)] The canonical and the dual connections are reduced to $(A^{\alpha}_{\mu\nu}, 0, {C}^{\alpha}_{\mu\nu})\text{ and }  (A^{\alpha}_{\nu\mu}, 0, {C}^{\alpha}_{\nu\mu}).$
\item[(f)] The horizontal counterpart of the canonical and the dual connections have the same form as in the classical AP-space. (However, in the  classical AP-context the $\lambda$'s are function of $x$ only whereas in the FP-Minkowskian space the $\lambda$'s are function of $x$ and $y$).
\item[(g)] The horizontal covariant derivatives with respect to the canonical and the dual connections have the same form as in the classical AP-space.
\end{description}
\end{theorem}

\begin{proof} We only prove $\textbf{(a), (b), (c), (e)}$. The rest is clear.
\begin{description}
  \item [(a)] Since the metric $g_{\mu \nu}(x, y)$ is a function of y only, it follows that $ \gamma^{\alpha}_{\mu \nu} = 0$. This implies that the canonical spray $G^{\alpha}=0$.
      Consequently, the Barthel connection $N^{\alpha}_{\beta}:= \dot{\partial}_{\beta} G^{\alpha}$ vanishes.
  \item[(b)] The $(v)hv$-Cartan torsion $\overcirc{P}\!^{\alpha}_{\mu \nu} =G^{\alpha}_{\mu\nu} -\, \overcirc{\Gamma}\!^{\alpha}_{\mu\nu}= y^{\beta} \overcirc{C}\!^{\alpha}_{\mu \nu{o \atop|}\beta} =0$ in all coordinate systems. This implies, in particular, that $G^{\alpha}_{\mu\nu} -\,\overcirc{\Gamma}\!^{\alpha}_{\mu\nu}=0$ in natural coordinate system, i.e.  $G^{\alpha}_{\mu\nu}=\, \overcirc{\Gamma}\!^{\alpha}_{\mu\nu}$ in natural coordinate system. But by $(a)\,G^{\alpha}_{\mu\nu}=0$ hence $\, \overcirc{\Gamma}\!^{\alpha}_{\mu\nu}=0$ in natural coordinate system.
 \item [(c)]This is because $\,G^{\alpha}_{\mu\nu}= \dot{\partial}_{\eta}N^{\alpha}_{\beta}=0$ (by \textbf{(a)}) and
 $\, \overcirc{\Gamma}\!^{\alpha}_{\mu\nu}=0$ (by \textbf{(b)}).
\item [(e)]The vanishing of $\overcirc{\Gamma}\!^{\alpha}_{\mu \nu}$ forces $A^{\alpha}_{\mu\nu}$ to coincide with ${\Gamma}^{\alpha}_{\mu\nu}$ i.e. $A^{\alpha}_{\mu\nu}={\Gamma}^{\alpha}_{\mu\nu}= {\undersym{\lambda}{i}^{\alpha}}(x,\,y)\, \partial_{\nu}\, {\undersym{\lambda}{i}_{\mu}}(x,\,y)$.
\end{description}
\end{proof}
The next proposition follows from the fact that the vanishing of a tensor in some coordinate system implies its vanishing in all coordinate systems. In more detail:
\begin{proposition}
Let $(M, \,\undersym{\lambda}{i}(x,\,y),\,F(x,\,y))$ be an FP-Minkowskian manifold. Then we have:
\begin{description}
\item[(a)]The vertical covariant derivatives for all connections of the FP-Minkowskian space remain the same as in the general case.
\item[(b)] The $(v)h$- torsion $R_{\mu\nu}^{\alpha}$ vanishes for all connections of the FP-Minkowskian space.
\item[(c)] All curvature and torsion tensors of the Berwald connection vanish.
\item[(d)] The only surviving torsion tensors of the Cartan connection is the $h(hv)$-torsion $\overcirc{C}\!^{\alpha}_{\mu\nu}$ whereas the only surviving curvature tensors is the $v$-curvature tensors $\overcirc{S}\!^{\alpha}_{\mu\nu\beta}$.
\end{description}
\end{proposition}
Unlike Theorem $4.10$, the above proposition holds in all coordinate systems.

\bigskip

The next table provides a comparison concerning the connections of an FP-Minkowskian manifold and their associated torsion and curvature tensors.
\vspace{8pt}
{\begin{center}{{Table 3: Torsion and curvature tensors of an FP-Minkowskian manifold}}
\\[0.4 cm]
{
\small\fontsize{12}{15}\selectfont{\begin{tabular} {|l|c|c|c|c|}\hline
 {\textbf{}}&{{\fontsize{12}{12}\selectfont Cartan }}&
 { {\fontsize{12}{12}\selectfont Canonical }} &{{\fontsize{12}{12}\selectfont Dual }}
\\[0.1 cm]\hline
\multirow{2}{*}{{(h)h-torsion} $T^{\alpha}_{\mu\nu}$}&\multirow{2}{*}{$0$}&\multirow{2}{*}{${T}^{\alpha}_{\mu\nu}$}&\multirow{2}{*}{$  -{T}^{\alpha}_{\mu\nu}$}\\[0.1cm]
\multirow{2}{*}{{(h)hv-torsion}
$ C^{\alpha}_{\mu\nu}$}&\multirow{2}{*}{$\overcirc{C}\!^{\alpha}_{\mu\nu} - B^{\alpha}_{\mu \nu}$} &
\multirow{2}{*}{${C}^{\alpha}_{\mu\nu}$}&
\multirow{2}{*}{${C}^{\alpha}_{\nu\mu}$}
\\[0.3 cm]
{(v)h-torsion} $R^{\alpha}_{\mu\nu}$&$0$
&$0$&
$0$
\\[0.1 cm]
{(v)hv-torsion}
$P^{\alpha}_{\mu\nu}$&$0$&
  $-{\Gamma}^{\alpha}_{\mu\nu}$&$ -{\Gamma}^{\alpha}_{\nu\mu}$
\\[0.1 cm]
{(v)v-torsion} $S^{\alpha}_{\mu\nu}$&$0$ &${S}^{\alpha}_{\mu\nu}$&$-{S}^{\alpha}_{\mu\nu}$
\\[0.1 cm]\hline
{h-curvature} $R^{\alpha^{\phantom{2}}}_{\mu\nu \sigma}$& $0$&$0$&
${T}^{\alpha}_{\nu \sigma |\mu}$
\\[0.1 cm]
{hv-curvature} $P^{\alpha}_{\mu\nu \sigma}$& $0$&
$0$&$\widetilde{P}^{\alpha}_{\mu\nu \sigma}$
\\[0.1 cm]
{ v-curvature}
$S^\alpha_{\mu\nu\sigma}$&$\overcirc{S}\!^\alpha_{\mu\nu\sigma}$&
$0$&${S}^{\alpha}_{\nu \sigma ||\mu}$
\\[0.1 cm]\hline


\end{tabular}}
}
\end{center}}

\vspace{.1cm}
In the above table, $\widetilde{P}^{\alpha}_{\nu\mu\sigma}$ is given by
 $$\widetilde{P}^{\alpha}_{\nu\mu\sigma} = {S}^{\alpha}_{\nu\mu |\sigma} +
{T}^{\alpha}_{\sigma\nu ||\mu} - {S}^{\eta}_{\mu\nu}
{T}^{\alpha}_{\sigma\eta} + {S}^{\alpha}_{\mu\eta}
{T}^{\eta}_{\sigma\nu} + {T}^{\alpha}_{\eta\nu} \, {C}^{\eta}_{\sigma\mu} -
{\Gamma}^{\eta}_{\sigma\mu} \, {S}^{\alpha}_{\eta\nu}.$$ Moreover, $\overcirc{S}\!^{\alpha}_{\mu\sigma\nu}$ remains as in the general case (Proposition 3.10).


\subsection{FP-Riemannian space}
\hspace{12pt}
We introduce the following definition.
\begin{definition}
An FP-manifold is said to be an FP-Riemannian space if the vertical counterpart ${C}^{\alpha}_{\mu\nu}$  of the canonical connection vanishes.
\end{definition}
As can be easily checked,
\begin{proposition}
Let $(M, \,\undersym{\lambda}{i}(x,\,y),\,F(x,\,y))$ be an FP-Riemannian space. Then the following hold:
\begin{description}
  \item[(a)] The $\lambda$'s are function of the positional argument $x$ only. Consequently, so are $\Gamma^{\alpha}_{\mu\nu} \text{ and } g_{\mu \nu}$ \footnote{ It should be noted that $g_{\mu\nu}$ being functions of $x$ does not necessarily imply that the $\lambda$'s are functions of $x$. The converse trivially holds.}.
  \item[(b)] The Cartan tensor vanishes.
  \item[(c)] The canonical spray is given by $G^{\alpha} = \frac{1}{2} \gamma^{\alpha}_{\mu\nu}(x)\, y^{\mu} \, y^{\nu}$. Consequently, $$ \,G^{\alpha}_{\sigma\beta}(x)= \gamma^{\alpha}_{\sigma\beta}(x)=\, \overcirc{\Gamma}\!^{\alpha}_{\sigma\beta}(x).$$
\end{description}
\end{proposition}


\begin{proposition}
Every FP-Riemannian space is an FP-Berwald space.
\end{proposition}
\begin{proof}
This follows directly from the fact that both ${\Gamma}^{\alpha}_{\mu\nu}$ and $\overcirc{\Gamma}\!^{\alpha}_{\mu\nu}$ are functions of $x$ only.
\end{proof}
\begin{theorem}Let $(M, \,\undersym{\lambda}{i}(x,\,y),\,F(x,\,y))$ be an FP-Riemannian space. Then
\begin{description}
  \item[(a)] The Cartan and Berwald connections coincide both reducing to $(\gamma^{\alpha}_{\mu\nu}(x), N^{\alpha}_{\mu}(x, y),0)$.
  \item[(b)] The canonical and dual connections reduce respectively to $({\Gamma}^{\alpha}_{\mu\nu}(x), N^{\alpha}_{\mu}(x, y),0)$ and \\
  $(\widetilde{\Gamma}^{\alpha}_{\mu\nu}(x), N^{\alpha}_{\mu}(x, y),0)$.
  \item[(c)] The horizontal contortion tensor $A^{\alpha}_{\mu \nu}$ is a function of the positional argument $x$ only whereas the vertical contortion tensor $B^{\alpha}_{\mu \nu}$ vanishes.
\end{description}
\end{theorem}
The torsion and curvature tensors of an FP-Riemmannian space are summarized in the following table.
\vspace{8pt}
{\begin{center}{{Table 4: Torsion and curvature tensors of an FP-Riemmannian space}}
\\[0.3 cm]
{
\small\fontsize{10}{15}\selectfont{\begin{tabular} {|l|c|c|c|c|c|}\hline
 {\textbf{}}&{{\fontsize{10}{10}\selectfont Cartan }}&
 { {\fontsize{10}{10}\selectfont Canonical }} &{{\fontsize{10}{10}\selectfont Dual }} &{{\fontsize{10}{10}\selectfont Berwald }}\\[0.1cm]\hline
\multirow{2}{*}{{(h)h-torsion} $T^{\alpha}_{\mu\nu}$}&\multirow{2}{*}{$0$} &\multirow{2}{*}{${T}^{\alpha}_{\mu\nu}$} &\multirow{2}{*}{$  -{T}^{\alpha}_{\mu\nu}$}&\multirow{2}{*}{$0$}\\[0.1cm]
\multirow{2}{*}{{(h)hv-torsion}
$ C^{\alpha}_{\mu\nu}$}&\multirow{2}{*}{$0$} &
\multirow{2}{*}{$0$}&\multirow{2}{*}{$0$}&\multirow{2}{*}{$0$}
\\[0.3 cm]
{(v)h-torsion} $R^{\alpha}_{\mu\nu}$&${R}^{\alpha}_{\mu\nu}$
&${R}^{\alpha}_{\mu\nu}$&
${R}^{\alpha}_{\mu\nu}$&${R}^{\alpha}_{\mu\nu}$\\[0.1 cm]
{(v)hv-torsion}
$\overcirc{P}\!^{\alpha}_{\mu\nu}$&$0$&
  ${P}^{\alpha}_{\mu\nu}$&${P}^{\alpha}_{\nu\mu}$&$0$
\\[0.1 cm]
{(v)v-torsion} $S^{\alpha}_{\mu\nu}$&$0$ &$0$&$0$&$0$
\\[0.1cm]\hline
{h-curvature} $R^{\alpha}_{\mu\nu \sigma}$& $\overcirc{R}\!^{\alpha}_{\mu\nu \sigma}$&
$0$&$T^{\alpha}_{\nu\sigma|\mu}$&$\bar{R}^{\alpha}_{\mu\nu \sigma}$
\\[0.1 cm]
{hv-curvature} $P^{\alpha}_{\mu\nu \sigma}$& $0$&$0$&$0$&$0$
\\[0.1 cm]
{v-curvture}
$S^\alpha_{\mu\nu\sigma}$&$0$&
$0$&$0$&$0$
\\[0.1 cm]\hline


\end{tabular}}
}
\end{center}}
In the above table the $h$-curvature tensor of the Cartan connection can be written in the following two forms
$${\overcirc{R}\!^{\alpha}_{\mu \nu \sigma}} := \mathfrak{U}_{(\nu,\sigma)} \{ \delta_{\sigma} \gamma^{\alpha}_{\mu \nu} + \gamma^{\eta}_{\mu \nu} \, \gamma^{\alpha}_{\eta \sigma}\}$$
 and
 $$\overcirc{R}\!^{\alpha}_{\mu \nu \sigma} = \mathfrak{U}_{(\sigma,\nu)} \{A^{\alpha}_{\mu \sigma | \nu} + A^{\eta}_{\mu \nu} \, A^{\alpha}_{\eta\sigma}\} +A^{\alpha}_{\mu \eta} {T}^{\eta}_{\sigma \nu}$$
 and the $h$-curvature tensor of the dual connection has the form $$ \bar{R}^{\alpha}_{\mu\nu \sigma}= \mathfrak{U}_{(\nu,\sigma)} \{{P}^{\alpha}_{\mu \nu | \sigma} + {P}^{\eta}_{\mu \nu} \, {P}^{\alpha}_{\eta\sigma}\} +{P}^{\alpha}_{\mu \eta} {T}^{\eta}_{\nu \sigma}.$$
\par
It should be noted that the horizontal geometric objects of the FP-Riemmannian space are identical to the corresponding geometric objects in classical AP-space. Consequently, the classical AP-space can be recovered from the FP-Riemmannian space.


\begin{flushleft}
    \textbf{Concluding remarks}
\end{flushleft}
\hspace{12pt}
In this paper, we have constructed a Finsler space from the building blocks of the GAP-space. We refer to the resulting space as a Finslerized Parallelizable  space or, in short, an FP-space. We conclude the paper by the following comments:
\begin{itemize}
  \item Five Finsler connections are studied with their associated torsion and curvature tensors, namely, the canonical, dual, Miron, Cartan and Berwald connections. The first three connections are related to the GAP-structure whereas the last two connections are related to the Finsler structure. Using the fact that the nonlinear connection of the GAP-space is the Barthel connection, we prove that the Miron connection coincides with the Cartan connection. Consequently, in an FP-space, we actually have only four different Finsler
       connections.

\item All curvature tensors of the different connections defined in the FP-space are derived. The vanishing of the three curvature tensors of the canonical connection enables us to express the other curvature tensors in terms of the torsion and contortion tensors of the canonical connection.
    This property is also satisfied in the context of classical AP-geometry; all curvature tensors of an AP-space can be expressed in terms of the torsion tensor of the AP-space.\footnote{In the classical AP-space, we have one torsion tensor, namely, $\Lambda^{\alpha}_{\mu\nu} = \Gamma^{\alpha}_{\mu\nu}- \Gamma^{\alpha}_{\nu\mu}$, where $\Gamma^{\alpha}_{\mu\nu}$ is the canonical connection. We refer to this torsion tensor as the torsion tensor of the AP-space.}

 \item Some special FP-spaces are investigated
 , namely, FP-Berwald, FP-Minkowskian and FP-Riemannian  manifolds. In the FP-Riemannian manifold the metric is a function of $x$ only (in all coordinate systems) whereas in the FP-Minkowskian manifold the metric is a function of $y$ only (in some coordinate system). Roughly speaking, the two cases represent two \lq \lq limiting" cases.

\item In the FP-Landsberg and the FP-Berwald manifolds the number of different connections is four whereas in the FP-Riemannian and the FP-Minkowskian manifolds it is reduced to three. In fact, the following proper inclusions hold: FP-Riemannian $\subset$ FP-Berwald $\subset$ FP-Landsberg and the FP-Minkowskian $\subset$ FP-Landsberg. In these special cases, many of the formulas obtained for the torsion and curvature tensors are considerably simplified. In particular, many geometric objects existing in the general case vanish.

\item Since an FP-space combines a Finsler structure and a GAP-structure, it is potentially more suitable for dealing with  physical phenomena that might escape the explanatory power of either GAP-geometry or Finsler geometry. Consequently, the FP-geometry may provide a better geometric structure for physical applications. Moreover, the structure of an FP-space would facilitate the following:
    \begin{description}
      \item[(a)] The use of the advantages of both machineries of GAP and Finsler geometries.
      \item[(b)] The gain of more information on the infrastructure of physical phenomena studied in AP-geometry.
      \item[(c)] Attribution of some physical meaning to objects of Finsler geometry, since AP-geometry
has been used in many physical applications.
      \item[(d)] Some of the problems of the general theory of relativity, such as the flatness of the rotation curves of spiral galaxies and some problems concerning the interpretation of the accelerated expansion of the Universe, are not satisfactorily explained in the context of Riemannian geometry. Theories in which the gravitational potential depends on both position and direction argument may be needed to overcome such difficulties. This is one of the aims motivating the present work.
    \end{description}

\item The paper is not intended to be an end in itself. In it, we try to construct a geometric structure suitable for dealing with and describing physical phenomena. The use of the classical AP-geometry and Finsler geometry in physical applications (for example \cite{Developments},\cite{geometry of Lag. sp.}) made us choose these two geometries as a guide line. In the present paper, we have focused on the study of the metric, some FP-connections and some special FP-spaces. The study of other geometric objects and special FP-spaces needed in physical applications will be given in the coming part of this work.
\end{itemize}
\vspace{5pt}
\noindent{Acknowledgment}

\vspace{3pt}
The authors would like to thank Professor M. I. Wanas for many useful discussions.


\bibliographystyle{plain}

\end{document}